\documentclass[11pt]{amsart}
\usepackage{amsmath}
\usepackage{amscd}
\usepackage{amssymb}
\usepackage{amsfonts}
\newtheorem{theorem}{Theorem}[section]
\newtheorem{lemma}[theorem]{Lemma}
\newtheorem{corollary}[theorem]{Corollary}
\newtheorem{proposition}[theorem]{Proposition}

\theoremstyle{definition}

\theoremstyle{remark}
\newtheorem{remark}[theorem]{Remark}
\numberwithin{equation}{section}

\begin{document}
\title[Time analyticity for heat equation on shrinkers]
{Time analyticity for heat equation on gradient shrinking Ricci solitons}

\author{Jia-Yong Wu}
\address{Department of Mathematics, Shanghai University, Shanghai 200444, China}
\email{jywu81@yahoo.com}

\date{\today}
\thanks{}

\begin{abstract}
On a complete non-compact gradient shrinking Ricci soliton, we prove the analyticity
in time for smooth solutions of the heat equation with quadratic exponential growth in
the space variable. This growth condition is sharp. As an application, we give a necessary
and sufficient condition on the solvability of the backward heat equation in a class of
functions with quadratic exponential growth on shrinkers.
\end{abstract}
\maketitle

\section{Introduction}\label{Int1}
As is well known, generic solutions of the heat equation are usually analytic in
space but not analytic in time. In Euclidean space, it is not difficult to
construct non-time-analytic solutions of the heat equation in a finite space-time
cylinder. It is therefore an interesting task to seek suitable conditions for ensuring
the time analyticity of solutions of the heat equation (see \cite{[Wi]}). In a
recent paper \cite{[Zh]}, Zhang proved that the ancient solutions with exponential
growth in the space variable are time-analytic on a complete non-compact Riemannian
manifold with the Ricci curvature bounded below. He also gave a necessary and
sufficient condition on the solvability of the backward heat equation in a class of
functions with exponential growth. Later, by choosing suitable space-time cutoff functions,
Dong and Zhang \cite{[DZ]} extended Zhang's results to the solutions with quadratic exponential
growth in the space variable. Meanwhile, they provided an example to indicate that the
growth condition is sharp. For more results about time analyticity for parabolic equations,
see \cite{[Ma]}, \cite{[KN]}, \cite{[Ko]}, \cite{[Gi]}, \cite{[EMZ]}, \cite{[HHW]} and
references therein.

In this paper, we will study the time analyticity for smooth solutions of the heat equation
on a complete non-compact gradient shrinking Ricci soliton (see the definition below).
We will prove that the analyticity in time always holds on a fixed gradient shrinking Ricci soliton
without any curvature assumption, provided that the solutions have quadratic exponential growth
in the space variable. This result may be useful for understanding the function theory
of gradient shrinking Ricci solitons.

Recall that an $n$-dimensional Riemannian manifold $(M, g)$ is called \emph{a gradient
shrinking Ricci soliton} $(M, g, f)$ (also called \emph{a shrinker} for short)
(see \cite{[Ham]}) if there exists a smooth potential function $f$ on $M$ such that
\begin{align}\label{Eq1}
\mathrm{Ric}+\mathrm{Hess}\,f=\frac 12g,
\end{align}
where $\text{Ric}$ is the Ricci curvature of $(M,g)$ and $\text{Hess}\,f$ is the Hessian
of $f$. Obviously, the flat Euclidean space $(\mathbb{R}^n, \delta_{ij})$ is a gradient
shrinking Ricci soliton with potential function $f=|x|^2/4$, which is called the Gaussian
shrinking Ricci soliton $(\mathbb{R}^n,\delta_{ij}, |x|^2/4)$. Shrinkers play an
important role in the Ricci flow \cite{[Ham]} and in Perelman's resolutions of the
Poincar\'e Conjecture \cite{[Pe],[Pe2],[Pe3]}, as they are self-similar solutions
and arise as limits of dilations of Type I singularities in the Ricci flow. Over
the past two decades, the geometric and analytic properties of shrinkers have
become an active issue, as this is useful for understanding the structure of manifolds
(see \cite{[Cao]} for an excellent survey).

From \eqref{Eq1}, Hamilton \cite{[Ham]} observed that
\[
C(g):=\mathrm{R}+|\nabla f|^2-f
\]
is a finite constant, where $\mathrm{R}$ is the scalar curvature of $(M,g)$.
Adding $f$ by a constant, without loss of generality, we assume that
\begin{equation}\label{Eq2}
\mathrm{R}+|\nabla f|^2=f.
\end{equation}
Under this normalization, it is not hard to see that (see also \cite{[CaNi]})
\begin{equation}\label{Eq3}
\int_M (4\pi)^{-\frac n2}e^{-f} dv=e^{\mu},
\end{equation}
where $dv$ is the Riemannian volume element on $(M,g)$ and $\mu=\mu(g,1)$ is the
entropy functional of Perelman \cite{[Pe]}. For the Ricci flow, Perelman's
entropy functional is time-dependent, but for a fixed gradient shrinking Ricci
soliton it is constant and finite.

Now we give the analyticity in time for smooth solutions of the heat equation on a complete
non-compact gradient shrinking Ricci soliton.
\begin{theorem}\label{main}
Let $(M,g, f)$ be an $n$-dimensional complete non-compact gradient shrinking Ricci soliton
satisfying \eqref{Eq1}, \eqref{Eq2} and \eqref{Eq3}. Let $u(x,t)$ be a smooth solution of
the heat equation $(\Delta-\partial_t)u=0$ on $M\times[-2, 0]$. For a fixed point $p\in M$,
if $u$ satisfies quadratic exponential growth in the space variable, i.e.,
\begin{equation}\label{expg}
|u(x,t)|\le A_1 e^{A_2d^2(x, p)}
\end{equation}
for all $(x,t)\in M\times[-2, 0]$, where $A_1$ and $A_2$ are some positive constants,
and $d(x,p)$ is the distance function from $p$ to $x$, then $u(x,t)$ is analytic
in time $t\in[-1,0]$ with radius $\delta>0$  depending only on $n$, $A_2$, $\mu$
and $f$. Moreover, we have that
\[
u(x,t)=\sum^\infty_{j=0}a_j(x)\frac{t^j}{j!},
\]
with $\Delta a_j(x)=a_{j+1}(x)$ and
\[
|a_j(x)|\le A_1e^{-\mu/2}e^{f(x)/2}(f(x)+1)^{n/4}
A_3^{j+1}j^j e^{2A_2 d^2(x,p)},\,\,\,j=0,1,2,\ldots,
\]
where $A_3$ is some constant depending only on $n$ and $A_2$,
and $0^0=1$. Here $\mu=\mu(g,1)$ denotes Perelman's
entropy functional.
\end{theorem}

\begin{remark}
The growth condition is necessary. As in
\cite{[DZ]}, let $v(x,t)$ be Tychonov's solution of the heat equation in
$(\mathbb{R}^n,\delta_{ij}, |x|^2/4)\times \mathbb{R}$ such that $v=0$ if
$t\le 0$ and $v$ is nontrivial for $t>0$. Then $u:=v(x, t+1)$ is a nontrivial
ancient solution of the heat equation and is not analytic in time. Note that
$|u(x,t)|$ grows faster than $e^{c |x|^2}$  for any $c>0$, but for any
$\epsilon>0$, $|u(x,t)|$ is bounded by $c_1e^{c_2|x|^{2+\epsilon}}$ for some
constants $c_1$ and $c_2$. This implies that growth condition \eqref{expg} is sharp.
\end{remark}

In general, the Cauchy problem to the backward heat equation is not solvable.
However, on a complete non-compact shrinker, we can obtain a solvable result by a simple
application of Theorem \ref{main}.
\begin{corollary}\label{cor}
Let $(M,g, f)$ be an $n$-dimensional complete non-compact gradient shrinking Ricci soliton
satisfying \eqref{Eq1}, \eqref{Eq2} and \eqref{Eq3}. For a fixed point $p\in M$, the Cauchy problem for the
backward heat equation
\begin{equation}\label{bhe}
\begin{cases}
(\Delta+\partial_t)u=0,\\
u(x,0)=a(x)
\end{cases}
\end{equation}
has a smooth  solution with quadratic exponential growth in $M\times(0,\delta)$ for some
$\delta>0$ if and only if
\begin{equation}\label{aijie}
|\Delta^j a(x)|\le e^{-\mu/2}e^{f(x)/2}(f(x)+1)^{n/4}
A_3^{j+1}j^j e^{A_4 d^2(x,p)}, \quad
j=0, 1,2,\ldots,
\end{equation}
where  $A_3$ and $A_4$ are some positive constants.
\end{corollary}

The rest of this paper is organized as follows: in Section \ref{sec2}, we recall some
properties of gradient shrinking Ricci solitons. In particular, we give a local mean
value type inequality on gradient shrinking Ricci solitons. In Section \ref{sec3},
adapting Dong-Zhang's proof strategy \cite{[DZ]}, we apply the mean value type
inequality of Section \ref{sec2} to prove Theorem \ref{main} and Corollary \ref{cor}.

\vspace{.1in}

\textbf{Acknowledgement}.
The author sincerely thanks Professor Qi S. Zhang for answering several questions
about the work \cite{[DZ]}. This work was partially supported by NSFC (11671141)
and NSFS (17ZR1412800).

\section{Some properties of shrinkers}\label{sec2}
In this section, we will present some basic propositions about complete non-compact
gradient shrinking Ricci solitons; these will be used in the proofs of our main results.

On an $n$-dimensional complete non-compact gradient shrinking Ricci soliton $(M,g,f)$
satisfying \eqref{Eq1}, \eqref{Eq2} and \eqref{Eq3}, from Chen's work (Proposition 2.2
in \cite{[Chen]}), we know that the scalar curvature is $\mathrm{R}\ge 0$.  Moreover, by
\cite{[PiRS]}, we know that the scalar curvature $\mathrm{R}$ must be strictly
positive, unless $(M,g,f)$ is the Gaussian shrinking Ricci soliton
$(\mathbb{R}^n,\delta_{ij}, |x|^2/4)$.

For any fixed point $p\in M$, by Theorem 1.1 of Cao-Zhou \cite{[CaZh]} (later refined by
Chow \emph{et al.} \cite{[Chowetc]}), we have
\[
\frac 14\left[\left(d(x,p)-2\sqrt{f(p)}-4n+\frac 43\right)_{+}\right]^2
\le f(x)\le \frac 14\left(d(x,p)+2\sqrt{f(p)}\right)^2
\]
for any $x\in M$, where $d(x,p)$ is the distance function from $p$ to $x$.
From this, $2\sqrt{f(x)}$ could be regarded as a distance-like function.
By the Gaussian shrinking Ricci soliton $(\mathbb{R}^n,\delta_{ij}, |x|^2/4)$,
the growth estimate of $f$ is sharp. Combining this estimate and
\eqref{Eq2}, for any point $p\in M$, we get that
\begin{equation}\label{upperR}
0\le \mathrm{R}(x)\le\frac 14\left(d(x,p)+2\sqrt{f(p)}\right)^2
\end{equation}
on $(M,g,f)$. It remains an interesting open question as to whether or not the scalar curvature
$\mathrm{R}$ is bounded from above by a uniform constant.

By Cao-Zhou \cite{[CaZh]} and Munteanu \cite{[Mun]}, the volume growth of a gradient
shrinking Ricci soliton can be regarded as an analogue of Bishop's theorem for
manifolds with non-negative Ricci curvature (see \cite{[MuWa2]}).
That is, there exists a constant $c(n)$ depending only on $n$ such that
\begin{equation}\label{upperV}
V_p(r)\le c(n)e^{f(p)}r^n
\end{equation}
for any $r>0$ and $p\in M$, where $V_p(r)$ denotes the volume of geodesic ball $B_p(r)$
with radius $r$ and center $p\in M$. From Haslhofer-M\"uller \cite{[HaMu]},
there exists a point $p_0\in M $ where $f$ attains its infimum such that $f(p_0)\leq n/2$.

In \cite{[LiWa]}, Li and Wang proved a local Sobolev inequality
on complete non-compact gradient shrinking Ricci solitons without any assumption.

\begin{lemma}\label{lem2}
Let $(M,g, f)$ be an $n$-dimensional complete gradient shrinking Ricci soliton
satisfying \eqref{Eq1}, \eqref{Eq2} and \eqref{Eq3}. Then, for each compactly supported
locally Lipschitz function $u(x)$ with support in $B_p(r)$, where $p\in M$ and $r>0$,
we have that
\begin{equation}\label{sobo}
\left(\int_{B_p(r)}u^{\frac{2n}{n-2}}dv\right)^{\frac{n-2}{n}}\le C(n)e^{-\frac{2\mu}{n}}\int_{B_p(r)}\left(4|\nabla u|^2+\mathrm{R} u^2\right) dv
\end{equation}
for some constant $C(n)$ depending only on $n$, where $\mu:=\mu(g,1)$ is the entropy functional
of Perelman, and $\mathrm{R}$ is the scalar curvature of $(M,g,f)$.
\end{lemma}
The proof of Lemma \ref{lem2} mainly depends on Perelman's entropy functional and
the Markov semigroup technique; see \cite{[LiWa]}. If the scalar curvature $\mathrm{R}$
is bounded, \eqref{sobo} is similar to a classical Sobolev inequality on compact manifolds
\cite{[Zh0]}. In this paper, we will apply the Sobolev inequality \eqref{sobo} to obtain
a local mean value type inequality on shrinkers.
\begin{proposition}\label{prop}
Let $(M,g, f)$ be an $n$-dimensional complete gradient shrinking Ricci soliton
satisfying \eqref{Eq1}, \eqref{Eq2} and \eqref{Eq3}. Fix $0<m<\infty$. Then there exists a
positive constant $C(n,m)$ depending on $n$ and $m$, such that, for any
$s\in \mathbb{R}$, and for any $0<\delta<1$, and for any smooth nonnegative solution $v$ of
\[
\left(\Delta-\partial_t\right)v(x,t)\ge0
\]
in the parabolic cylinder $Q_r(p,s):=B_p(r)\times[s-r^2,s]$, where $p\in M$ and $0<r<2$,
we have
\begin{equation}\label{mean}
\sup_{Q_{\delta r}(p,s)}\{v^m\}
\leq \frac{C(n,m)(\mathrm{R}_{\mathrm{M}}+1)^{n/2}}{(1-\delta)^{2+n}\,e^{\mu}\,r^{2+n}}\int_{Q_r(p,s)}v^m \ dx dt,
\end{equation}
where $\mathrm{R}_{\mathrm{M}}:=\sup_{x\in B_p(r)}{\mathrm{R}(x)}$
and $\mu:=\mu(g,1)$ is Perelman's entropy functional.
\end{proposition}

\begin{remark}
It should be noted that our local mean value type inequality \emph{only} holds for
a local geodesic ball (here we choose the radius $0<r<2$), and does not hold
for a geodesic ball of any radius.
\end{remark}

The point $(p,s)$ and the radius $r$ in Proposition \ref{prop} is customarily called the
vertex and the size of parabolic cylinder $Q_r(p,s)$, respectively. Compared with the
classical mean value type inequality of manifolds, there seems to be a lack of a volume
factor $V_p(r)$ in \eqref{mean}. However if the factor $r^n$ is regarded as
$V_p(r)$, this inequality is very similar to the classical case.

\begin{proof}[Proof of Proposition \ref{prop}]
Following the argument of \cite{[WuWu]}, the proof technique used here is the delicate
Moser iteration applied to local Sobolev inequality \eqref{sobo}. It is emphasized
that the explicit coefficients of the mean value inequality in terms of the
Sobolev constants in \eqref{sobo} should be carefully examined.

We first prove \eqref{mean} for $m=2$. Given a nonnegative smooth solution $u$ of
$(\Delta-\partial_t)v\ge 0$ for any nonnegative function $\phi\in C^{\infty}_0(B)$,
where $B:=B_p(r)$, $p\in M$ and $r>0$, we have that
\[
\int_B(\phi u_t+\nabla\phi\nabla u)dv\le 0.
\]
Set $\phi=\psi^2u$, $\psi\in C^{\infty}_0(B)$. Then
\begin{equation*}
\begin{aligned}
\int_B(\psi^2 uu_t+\psi^2|\nabla u|^2)dv&\le 2\left|\int_B u\psi \nabla u\nabla\psi dv\right|\\
&\le 2\int_B|\nabla\psi|^2u^2dv+\frac 12\int_B\psi^2|\nabla u|^2dv,
\end{aligned}
\end{equation*}
so we get that
\[
\int_B(2\psi^2 uu_t+|\nabla(\psi u)|^2)dv\le 4\,\|\nabla\psi\|^2_{\infty} \int_{\mathrm{supp}(\psi)} u^2dv.
\]
Multiplying a smooth function $\lambda(t)$, which will be determined later, from the above
inequality, we have that
\begin{equation}
\begin{aligned}\label{basinequ}
\partial_t\left(\int_B(\lambda\psi u)^2dv\right)+&\lambda^2\int_B|\nabla(\psi u)|^2dv\\
&\leq C\lambda\Big(\lambda\|\nabla \psi\|^2_{\infty}+|\lambda'|\sup\psi^2\Big)\int_{\mathrm{supp}(\psi)} u^2dv,
\end{aligned}
\end{equation}
where $C$ is finitely constant, though this may change from line to line in the ensuing computations.

We choose $\psi$ and $\lambda$ such that, for $0<\sigma'<\sigma<1$, $\kappa=\sigma-\sigma'$,
\begin{enumerate}
\item
$0\leq\psi\leq 1$, $\mathrm{supp}(\psi)\subset\sigma B$, $\psi=1$ in $\sigma' B$ and
$|\nabla\psi|\le 2(\kappa r)^{-1}$, where $\sigma B:=B_p(\sigma r)$;
\item
$0\leq\lambda\leq 1$, $\lambda=0$ in $(-\infty,s-\sigma r^2)$,  $\lambda=1$ in $(s-\sigma' r^2,+\infty)$, and
$|\lambda'(t)|\le 2(\kappa r)^{-2}$.
\end{enumerate}
Set $I_\sigma:=[s-\sigma r^2,s]$. For any $t\in I_{\sigma'}$, integrating \eqref{basinequ} over
$[s-r^2, t]$,
\begin{equation}\label{integso}
\sup_{I_{\sigma'}}\left\{\int_B\psi u^2dv\right\}+\int_{B\times I_{\sigma'}}|\nabla(\psi u)|^2dv dt
\le C(\kappa r)^{-2}\int_{{\sigma B}\times I_{\sigma}} u^2dv dt,
\end{equation}
On the other hand, by H\"older's inequality and Lemma \ref{lem2}, we have that
\begin{equation}
\begin{aligned}\label{space-timeineq}
\int_{\overline{B}}\varphi^{2(1+2/n)}dv
&\le\left(\int_{\overline{B}}\varphi^2dv\right)^{2/n}
\left(\int_{\overline{B}}|\varphi|^{\frac{2n}{n-2}}dv\right)^{\frac{n-2}{n}}\\
&\le \left(\int_{\overline{B}}\varphi^2dv\right)^{2/n}
\left[C(n)e^{-\frac{2\mu}{n}}\int_{\overline{B}}(4|\nabla \varphi|^2+\mathrm{R}\,\varphi^2)dv\right]
\end{aligned}
\end{equation}
for all $\varphi\in C^{\infty}_0(\overline{B})$. Combining \eqref{integso} and \eqref{space-timeineq}, we
finally get that
\[
\int_{{\sigma' B}\times I_{\sigma'}}u^{2\theta}dv dt
\leq E(B)\left(C(\kappa r)^{-2}\int_{{\sigma B}\times I_{\sigma}} u^2dv dt\right)^\theta,
\]
with $\theta=1+2/n$, where $E(B):=C(n)e^{-\frac{2\mu}{n}}\left(\mathrm{R}_{\mathrm{M}}+1\right)$
and $\mathrm{R}_{\mathrm{M}}:=\sup_{x\in B_p(r)}{\mathrm{R}(x)}$. We would like to point out that we have used the condition $0<r<2$ in the above inequality.

For any $m\geq 1$, $u^m$ is also a nonnegative solution of $(\Delta-\partial_t)u\ge 0$.
Hence the above inequality indeed implies that
\begin{equation}\label{intds2}
\int_{{\sigma' B}\times I_{\sigma'}}u^{2m\theta}dv dt
\leq E(B)\left(C(\kappa r)^{-2}\int_{{\sigma B}\times I_{\sigma}} u^{2m}dv dt\right)^\theta
\end{equation}
for $m\geq1$.

Let $\kappa_i=(1-\delta)2^{-i}$, which satisfies $\Sigma^{\infty}_1\kappa_i=1-\delta$.
Let $\sigma_0=1$ and $\sigma_{i+1}=\sigma_i-\kappa_i=1-\Sigma^i_1\kappa_j$.
Applying \eqref{intds2} for $m=\theta^i$, $\sigma=\sigma_i$, $\sigma'=\sigma_{i+1}$,
we have that
\[
\int_{{\sigma_{i+1} B}\times I_{\sigma_{i+1}}}u^{2\theta^{i+1}}dv dt
\leq E(B)\left\{\frac{C^{i+1}}{\left[(1-\delta)r\right]^2}\int_{{\sigma_i B}\times I_{\sigma_i}} u^{2\theta^i}dv dt\right\}^\theta.
\]
Therefore,
\begin{equation*}
\begin{aligned}
\left(\int_{{\sigma_{i+1} B}\times I_{\sigma_{i+1}}}u^{2\theta^{i+1}}dv dt\right)^{\theta^{-i-1}}
\le \frac{C^{\Sigma j\theta^{1-j}}E(B)^{\Sigma\theta^{-j}}}{\left[(1-\delta)r\right]^{2\Sigma\theta^{1-j}}}
\int_{Q_r(p,s)} u^2dv dt,
\end{aligned}
\end{equation*}
where $\Sigma$ denotes the summations from $1$ to $i+1$. Letting $i\to \infty$,
\begin{equation}\label{prmi}
\sup_{{\delta B}\times I_\delta}\{u^2\}\le \frac{C E(B)^{n/2}}{[(1-\delta)r]^{2+n}}\|u\|^2_{2,Q_r(p,s)},
\end{equation}
which clearly implies \eqref{mean} when $m=2$, since $I_{\delta^2}\subset I_{\delta}$.

The case $m>2$ then follows by the case $m=2$, because if $u$ is a nonnegative
solution of $(\Delta-\partial_t)v\ge 0$, then
$u^m$, $m\geq1$, is also a nonnegative solution of $(\Delta-\partial_t)v\ge 0$.
All in all, we do, in fact, prove \eqref{mean} when $m\geq 2$.

When $0<m<2$, we will apply \eqref{prmi} to prove \eqref{mean} by a different iterative
argument. Let $\sigma\in (0,1)$ and $\rho=\sigma+(1-\sigma)/4$. Then \eqref{prmi} implies
that
\[
\sup_{{\sigma B}\times I_\sigma}\{u\}\le \frac{F(B)}{(1-\sigma)^{1+n/2}}\|u\|_{2,{\rho B}\times I_\rho},
\]
where $F(B):=C(n)e^{-\mu/2}\left(\mathrm{R}_{\mathrm{M}}+1\right)^{n/4}r^{-1-n/2}$.
Using
\[
\|u\|_{2,Q}\le \|u\|^{1-m/2}_{\infty,Q}\cdot\|u\|^{m/2}_{m,Q},\quad 0<m<2
\]
for any parabolic cylinder $Q$, we then have that
\begin{equation}\label{moserit}
\|u\|_{\infty,{\sigma B}\times I_\sigma}\le \frac{G(B)}{(1-\sigma)^{1+n/2}}\|u\|^{1-m/2}_{\infty,{\rho B}\times I_\rho},
\end{equation}
where $G(B):=F(B)\|u\|^{m/2}_{m,Q_r(p,s)}$.

Now fix $\delta\in (0,1)$ and let $\sigma_0=\delta$ and $\sigma_{i+1}=\sigma_i+(1-\sigma_i)/4$,
which satisfies that $1-\sigma_i=(3/4)^i(1-\delta)$.
Applying \eqref{moserit} to $\sigma=\sigma_i$ and $\rho=\sigma_{i+1}$ for each $i$,
we have that
\[
\|u\|_{\infty,{\sigma_i B}\times I_{\sigma_i}}\le (4/3)^{i(1+n/2)}
\frac{G(B)}{(1-\delta)^{1+n/2}}\,\|u\|^{1-m/2}_{\infty,{\sigma_{i+1} B}\times I_{\sigma_{i+1}}}.
\]
Therefore, for any $i$,
\[
\|u\|_{\infty,{\delta B}\times I_\delta}\le (4/3)^{(1+n/2)\Sigma j(1-m/2)^j}
\left[\frac{G(B)}{(1-\delta)^{1+n/2}}\right]^{\Sigma(1-m/2)^j}
\|u\|^{(1-m/2)^i}_{\infty,{\sigma_i B}\times I_{\sigma_i}},
\]
where $\Sigma$ denotes the summations from $0$ to $i-1$. Letting $i\to \infty$,
\[
\|u\|_{\infty,{\delta B}\times I_\delta}\le(4/3)^{\frac{2-m}{m^2}(2+n)}
\left[\frac{G(B)}{(1-\delta)^{1+n/2}}\right]^{2/m},
\]
which implies \eqref{mean} for $0<m<2$, by $I_{\delta^2}\subset I_{\delta}$ and the definition of $G(B)$.
\end{proof}

\section{Proof of results}\label{sec3}
In this section, adapting the argument of Dong-Zhang \cite{[DZ]},
we will apply the preceding propositions of shrinkers in Section \ref{sec2} to prove Theorem
\ref{main} and Corollary \ref{cor}. We first prove Theorem \ref{main}.

\begin{proof}[Proof of Theorem \ref{main}]
Without loss of generality, we may assume that $A_1=1$, because the heat equation is linear.
To prove the theorem, it suffices to confirm the result at space-time point $(x,0)$
for any $x \in M$.

Since $u$ is a given smooth solution to the heat equation, $u^2$ is a nonnegative
subsolution to the heat equation. Given a point $x_0 \in M$ and a positive
integer $k$, by letting $s=0$, $r=1/\sqrt k$, $m=1$,
$\delta=1/2$ in the mean value type inequality of Proposition \ref{prop}, we have
that

\begin{equation}
\begin{aligned}\label{genemeanieq}
\sup_{Q_{1/(2 \sqrt{k})} (x_0,0)} u^2 &\le \frac{C_1(n)
\left[\frac 14\left(\frac{1}{\sqrt k}+2\sqrt{f(x_0)}\right)^2+1\right]^{n/2}}{e^{\mu}(1/\sqrt k)^{2+n}}
\int_{Q_{1/\sqrt{k}}(x_0,0)} u^2(x,t) dv dt\\
&\le C_2(n)e^{-\mu}k^{n/2+1}(f(x_0)+1)^{n/2}\int_{Q_{1/\sqrt{k}}(x_0,0)} u^2(x,t) dv dt,
\end{aligned}
\end{equation}
where constants $C_1(n)$ and $C_2(n)$ both depend only on $n$. Here, in the
first inequality above, we used the estimate
\[
\sup_{B_{x_0}(1/\sqrt k)}\mathrm{R}\le\frac 14\left(\frac{1}{\sqrt k}+2\sqrt{f(x_0)}\right)^2,
\]
and in the second inequality we only used the simple fact that the size of the cubes
is less than one.

Since $\partial^k_t u$ is also a solution to the heat equation, then substituting
this into \eqref{genemeanieq} gives
\begin{equation}\label{mviqdkt}
\sup_{Q_{1/(2 \sqrt{k})} (x_0,0)} (\partial^k_t u)^2
\le C_2(n)e^{-\mu}k^{n/2+1}(f(x_0)+1)^{n/2}\int_{Q_{1/\sqrt{k}}(x_0,0)}(\partial^k_t u)^2(x,t) dv dt.
\end{equation}

In that follows, following we will bound the right hand side of \eqref{mviqdkt}.
For integers $j=1, 2, \ldots, k,$ consider the domains
\[
\Omega^1_{j}=B_{x_0}\left(\tfrac{j}{\sqrt{k}}\right)\times\left[-\tfrac{j}{k},0\right],\quad
\Omega^2_j=B_{x_0}\left(\tfrac{j+0.5}{\sqrt{k}}\right)\times\left[-\tfrac{j+0.5}{k},0\right].
\]
Obviously, $\Omega^1_j \subset \Omega^2_j \subset \Omega^1_{j+1}$.
Let $\psi^{(1)}_j$ be a standard Lipschitz cutoff function supported in
\[
B_{x_0}\left(\tfrac{j+0.5}{\sqrt{k}}\right)\times\left[-\tfrac{j+0.5}{k},\tfrac{j+0.5}{k}\right]
\]
such that
\[
\psi^{(1)}_j=1\,\,\, \mathrm{in}\,\,\, \Omega^1_j
\qquad \mathrm{and}\qquad
\left|\nabla\psi^{(1)}_j\right|^2 + \left|\partial_t \psi^{(1)}_j\right|\le C k,
\]
where $C$ is a universal constant (this includes the following constants $C$)
that may be changed line by line.

For the above cutoff function $\psi=\psi^{(1)}_j$, multiplying $(u_t)^2=u_t\Delta u$
by $\psi^2$, integrating it over $\Omega^2_j$, and using integration by parts,
we compute that
\begin{equation*}
\begin{aligned}
\int_{\Omega^2_j}(u_t)^2\psi^2dx dt&=\int_{\Omega^2_j}u_t\Delta u\psi^2dvdt\\
&=-\int_{\Omega^2_j} ((\nabla u)_t \nabla u) \psi^2dvdt
- \int_{\Omega^2_j} u_t \nabla u \nabla \psi^2 dvdt\\
&=-\frac 12\int_{\Omega^2_j}(|\nabla u |^2)_t\psi^2dv dt-2\int_{\Omega^2_j}u_t\psi\nabla u\nabla\psi dvdt\\
&\le\frac 12\int_{\Omega^2_j}|\nabla u|^2(\psi^2)_tdvdt+\frac 12\int_{\Omega^2_j}(u_t)^2 \psi^2 dvdt+2\int_{\Omega^2_j}
 |\nabla u|^2|\nabla \psi|^2 dvdt\\
&\le Ck\int_{\Omega^2_j}|\nabla u|^2dvdt+\frac 12\int_{\Omega^2_j}(u_t)^2 \psi^2 dvdt,
\end{aligned}
\end{equation*}
where we used the property of cutoff function $\psi$ in the fifth line.
Using $\Omega^1_j \subset \Omega^2_j$ and the property of $\psi$, the above inequality implies that
\begin{equation}\label{omg1j<}
\int_{\Omega^1_j}(u_t)^2 dvdt\le Ck\int_{\Omega^2_j}|\nabla u |^2dv dt.
\end{equation}

Let $\psi^{(2)}_j$ also be a standard Lipschitz cutoff function supported in
\[
B_{x_0}\left(\tfrac{j+1}{\sqrt{k}}\right)\times\left[-\tfrac{j+1}{k},\tfrac{j+1}{k}\right]
\]
such that $\psi^{(2)}_j=1$ in $\Omega^2_j$ and
$\left|\nabla\psi^{(2)}_j\right|^2+\left|\partial_t\psi^{(2)}_j\right|\le Ck$.
Then we can apply the standard Caccioppoli inequality (energy estimate)
between the cubes $\Omega^2_j$ and $\Omega^1_{j+1}$ to obtain that
\begin{equation}\label{energyest}
\int_{\Omega^2_j}|\nabla u|^2dvdt\le C k\int_{\Omega^1_{j+1}}u^2dvdt.
\end{equation}
Indeed, for the cutoff function $\varphi=\psi^{(2)}_j$, multiplying $uu_t=u\Delta u$
by $\varphi^2$, integrating it over $\Omega^1_{j+1}$, and using integration by parts,
we have that
\begin{equation*}
\begin{aligned}
\frac 12\int_{\Omega^1_{j+1}}\partial_t(u^2\varphi^2)dvdt
-\int_{\Omega^1_{j+1}}&\varphi\varphi_t u^2dvdt\\
&\quad=\int_{\Omega^1_{j+1}}uu_t\varphi^2dv dt\\
&\quad=\int_{\Omega^1_{j+1}}u\Delta u\varphi^2dvdt\\
&\quad=-\int_{\Omega^1_{j+1}}|\nabla u|^2\varphi^2dvdt-2\int_{\Omega^1_{j+1}}u \nabla u\varphi \nabla\varphi dvdt.
\end{aligned}
\end{equation*}
Observing that the first term of left hand side is
\[
\frac 12\int_{\Omega^1_{j+1}}\partial_t(u^2\varphi^2)dvdt
=\frac 12\int_{B_{x_0}\left(\tfrac{j+1}{\sqrt{k}}\right)}u^2(x,0)dv\ge 0,
\]
we conclude that
\begin{equation*}
\begin{aligned}
\int_{\Omega^1_{j+1}}|\nabla u|^2\varphi^2dvdt&\le
\int_{\Omega^1_{j+1}}\varphi\varphi_t u^2dvdt-2\int_{\Omega^1_{j+1}}u\nabla u\varphi\nabla\varphi dvdt\\
&\le\int_{\Omega^1_{j+1}}\varphi\varphi_t u^2dvdt+\frac 12\int_{\Omega^1_{j+1}}|\nabla u|^2\varphi^2dvdt+2\int_{\Omega^1_{j+1}}u^2|\nabla\varphi|^2 dvdt,
\end{aligned}
\end{equation*}
where we have used Young's inequality in the second inequality above. Therefore,
\[
\int_{\Omega^1_{j+1}}|\nabla u|^2\varphi^2dvdt
\le2\int_{\Omega^1_{j+1}}\varphi\varphi_t u^2dvdt+4\int_{\Omega^1_{j+1}}u^2|\nabla\varphi|^2 dvdt.
\]
Then \eqref{energyest} follows by $\Omega^2_j \subset \Omega^1_{j+1}$ and the property of cut-off function $\varphi$.

Combining \eqref{energyest}  and \eqref{omg1j<} yields that
\[
\int_{\Omega^1_j}(u_t)^2 dvdt\le C k^2\int_{\Omega^1_{j+1}}u^2dvdt.
\]
Since $\partial^j_t u$ is also a solution of the heat equation, we can replace $u$ in
the above inequality by $\partial^j_t u$ to deduce, by induction, that
\[
\int_{\Omega^1_1}(\partial^k_t u)^2dvdt\le C^k k^{2k}\int_{\Omega^1_k}u^2dvdt.
\]

Noticing that ${\Omega}^1_1= Q_{1/\sqrt{k}}(x_0,0)$ and
${\Omega}^1_k=B_{x_0}(\sqrt{k})\times[-1,0]$, we substitute
the above inequality into \eqref{mviqdkt} to get that
\begin{equation}\label{pingh}
\sup_{Q_{1/(2 \sqrt{k})}(x_0,0)}(\partial^k_t u)^2(x,t)\le C_2(n)e^{-\mu}k^{n/2+1}
(f(x_0)+1)^{n/2}C^kk^{2 k}\int_{B_{x_0}(\sqrt{k})\times[-1, 0]}u^2dvdt.
\end{equation}
Using quadratic exponential growth condition \eqref{expg} and the volume growth of
shrinker \eqref{upperV}, \eqref{pingh} can be further simplified as
\[
\sup_{Q_{1/(2\sqrt{k})}(x_0,0)}\left|\partial^k_tu(x,t)\right|\le C_2(n)e^{-\mu/2}k^{n/4+1/2}
(f(x_0)+1)^{n/4}C^{k/2}k^ke^{f(x_0)/2}k^{n/4}e^{A_2d^2(\xi,p)}
\]
for some point $\xi\in B_{x_0}(\sqrt{k})$ and for all integers $k \ge 1$.
By the triangle inequality,
\[
d(\xi,p)\le d(x_0, p)+d(x_0,\xi)\le d(x_0,p)+k^{1/2},
\]
so
\[
d^2(\xi,p)\le (d(x_0,p)+k^{1/2})^2\le 2d^2(x_0,p)+2k.
\]
Therefore,
\begin{equation}
\begin{aligned}\label{djtu}
\left|\partial^k_t u(x,t)\right|&\le C_2(n)e^{-\mu/2}e^{f(x_0)/2}(f(x_0)+1)^{n/4}
k^{n/2+1/2} C^{k/2}k^k e^{2A_2 k}\, e^{2A_2 d^2(x_0,p)}\\
&\le e^{-\mu/2}e^{f(x_0)/2}(f(x_0)+1)^{n/4} A_3^{k+1}\,k^k\, e^{2A_2 d^2(x_0,p)}
\end{aligned}
\end{equation}
for all $(x,t)\in Q_{1/(2\sqrt{k})}(x_0,0)$, and
for all integers $k\ge1$, where $A_3$ is a positive constant depending only on
$n$, $C$ and $A_2$.

Now fix a real number $R \ge 1$. For any point $x\in B_p(R)$, we choose a
positive integer $j$  and $t\in [-\delta, 0]$
for some small $\delta>0$. By Taylor's theorem,
\begin{equation}\label{jtaylor}
u(x,t)-\sum^{j-1}_{i=0}\partial^i_tu(x,0)\frac{t^i}{i!}=\frac{t^j}{j!}\partial^j_su(x,s),
\end{equation}
where $s=s(x,t,j)\in[t,0]$.
Using \eqref{djtu}, we know that, for sufficiently small $\delta>0$, which depends
on $n$, $A_2$, $A_3$, $\mu$ and $f$, the
right hand side of \eqref{jtaylor} converges to $0$ uniformly
for $x \in B_p(R)$ as $j \to \infty$. Hence
\[
u(x,t)=\sum^{\infty}_{j=0}\partial^j_t u(x,0)\frac{t^j}{j!}.
\]
Thus $u(x,t)$ is analytic in time $t$ with radius $\delta$. Set
$a_j= a_j(x)=\partial^j_t u(x,0)$. By \eqref{djtu} again, we have
that
\[
\partial_tu(x,t)=\sum^{\infty}_{j=0}a_{j+1}(x)\frac{t^j}{j!}\quad\mathrm{and}\quad
\Delta u(x,t)=\sum^{\infty}_{j=0}\Delta a_j(x)\frac{ t^j}{j!},
\]
where both series converge uniformly for $(x,t) \in B_p(R)\times[-\delta, 0]$
for any fixed $R>0$. Since $u(x,t)$ is a solution of the heat equation,
this implies that
\[
\Delta a_j(x)=a_{j+1}(x),
\]
with
\[
|a_j(x)|\le A_3^{j+1}e^{-\mu/2}e^{f(x)/2}(f(x)+1)^{n/4}j^je^{2A_2d^2(x,p)},
\]
where $A_3$ is a positive constant depending only on $n$ and $A_2$.
\end{proof}

In the end, we apply Theorem \ref{main} to prove Corollary \ref{cor}.

\begin{proof}[Proof of Corollary \ref{cor}]
Assume that $u(x,t)$ is a smooth solution to the Cauchy problem of the backward
heat equation \eqref{bhe} with quadratic exponential growth. Then $u(x,-t)$ is
also a smooth solution of the heat equation with quadratic exponential growth.
By Theorem \ref{main}, we have that
\[
u(x,-t)=\sum^\infty_{j=0}\Delta^ja(x)\frac{(-t)^j}{j!}.
\]
Then \eqref{aijie} follows by letting $\Delta^ja(x)=a_j(x)$ in the theorem.

On the other hand, we assume that \eqref{aijie} holds. We then claim that
\[
u(x,t)=\sum^\infty_{j=0}\Delta^ja(x)\frac{t^j}{j!}
\]
is a smooth solution to the heat equation for $t\in[-\delta,0]$ with some constant
$\delta>0$ sufficiently small. Indeed, \eqref{aijie} guarantees that the above series
and the two series
\[
 \sum^\infty_{j=0}\Delta^{j+1}a(x)\frac{t^j}{j!}\quad\text{and}\quad
\sum^\infty_{j=0}\Delta^{j}a(x)\frac{\partial_t t^j}{j!}
\]
all converge absolutely and uniformly in $[-\delta, 0]\times B_p(R)$
for any fixed $R>0$. Hence $(\Delta-\partial_t)u(x,t)=0$,
and the claim follows. Moreover, we observe that
\begin{equation*}
\begin{aligned}
|u(x,t)|&\le\sum^\infty_{j=0}\left|\Delta^j a(x)\right|\frac{|t|^j}{j!}\\
&\le e^{-\mu/2}e^{f(x)/2}(f(x)+1)^{n/4}A_3
e^{A_4 d^2(x,p)} \sum^\infty_{j=0} \frac{(A_3\,j|t|)^j}{j!}\\
&\leq e^{-\mu/2}e^{f(x)/2}(f(x)+1)^{n/4}A_3A_5e^{A_4 d^2(x,p)},
\end{aligned}
\end{equation*}
provided that $t\in[-\delta,0]$ with some sufficiently small constant $\delta>0$,
where we used the fact that the series $\sum^\infty_{j=0} \tfrac{(A_3\,j|t|)^j}{j!}$
converges in $[-\delta,0]$ and its summation is no more than some constant $A_5>0$;
that is, $u(x,t)$ has quadratic exponential growth. Hence, $u(x, -t)$ is a solution
to the Cauchy problem of the backward heat equation \eqref{bhe} of quadratic exponential
growth.
\end{proof}

\bibliographystyle{amsplain}

\end{document}